\documentclass[12pt]{amsart}
\usepackage{amsthm, amssymb}
\usepackage{amsfonts, amscd}
\usepackage{epsfig,multicol}
\theoremstyle{plain} \numberwithin{equation}{section}
\newtheorem{thm}{Theorem}[section]
\newtheorem{cor}[thm]{Corollary}
\newtheorem{prop}[thm]{Proposition}
\newtheorem{lem}[thm]{Lemma}

\theoremstyle{definition}

\newtheorem{exam}[thm]{Example}
\newtheorem*{rem}{Remark}

\def\Z{\Bbb Z}
\def\C{\Bbb C}
\def\R{\Bbb R}
\def\CO{\Lambda}
\def\SA{S}
\def\P{S}

 \DeclareMathOperator{\GL}{GL}
\DeclareMathOperator{\Aut}{Aut}

\def\diffeomorphism{homeomorphism}
\def\diffeomorphic{homeomorphic}

\begin{document}
\title[2-torus manifolds]{\large \bf Equivariant classification of
2-torus manifolds}
\author[Zhi L\"u and Mikiya Masuda ]{Zhi L\"u and Mikiya Masuda}
\footnote[0]{
\endgraf
 {\bf 2000AMS Classification:} 57S10,  14M25, 52B70.
 \endgraf
 The first author is supported by grants from NSFC (No. 10371020 and No. 10671034) and JSPS (No. P02299).}
\address{Institute of Mathematics, School of Mathematical Science, Fudan University, Shanghai,
200433, P.R. China.} \email{zlu@fudan.edu.cn}

\address{Department of Mathematics, Osaka City
University, Sumiyoshi-ku, Osaka 558-8585, Japan.}
\email{masuda@sci.osaka-cu.ac.jp}


\maketitle


\section{Introduction}

In this paper, we consider the equivariant classification of locally
standard 2-torus manifolds. A 2-torus manifold is a closed smooth manifold
 of dimension $n$ with an effective action of a 2-torus group
$(\Z_2)^n$ of rank $n$, and it is said to be locally standard if it
is locally isomorphic to a faithful representation of $(\Z_2)^n$ on
$\R^n$. The orbit space $Q$ of a locally standard 2-torus $M$ by the
action is a nice manifold with corners. When $Q$ is a simple convex
polytope, $M$ is called a small cover and studied in \cite{dj}. A
typical example of a small cover is a real projective space $\R P^n$
with a standard action of $(\Z_2)^n$. Its orbit space is an
$n$-simplex. On the other hand, a typical example of a compact
non-singular toric variety is a complex projective space $\C P^n$
with a standard action of $(\C^*)^n$ where $\C^*=\C\backslash\{0\}$.
$\C P^n$ has complex conjugation and its fixed point set is $\R
P^n$.  More generally, any compact non-singular toric variety admits
complex conjugation and its fixed point set often provides an
example of a small cover.
Similarly to the theory of toric varieties, an interesting
connection among topology, geometry and combinatorics is discussed
for small covers in \cite{dj}, \cite{djs} and \cite{ga-sc02}.
Although locally standard 2-torus manifolds form a much wider class
than small covers, one can still expect such a connection.
See \cite{l} for the study of 2-torus manifolds
from the viewpoint of cobordism.

\vskip .2cm

The orbit space $Q$ of a locally standard 2-torus manifold $M$
contains a lot of topological information on $M$. For instance, when
$Q$ is a simple convex polytope (in other words, when $M$ is a small
cover), the betti numbers of $M$ (with $\Z_2$ coefficient) are
described in terms of face numbers of $Q$ (\cite{dj}). This is not
the case for a general $Q$, but the euler characteristic of $M$ can
be described in terms of $Q$ (Theorem~\ref{euler}). Although $Q$
contains a lot of topological information on $M$, $Q$ is not
sufficient to reproduce $M$, i.e., there are many locally standard
2-torus manifolds with the same orbit space in general. We need two
data to reproduce $M$ from $Q$.  One is a characteristic function on
$Q$ introduced in \cite{dj}. It is a map from the set of
codimension-one faces of $Q$ to $(\Z_2)^n$ satisfying a certain
linearly independence condition. Roughly speaking, a characteristic
function provides information on the set of non-free orbits in $M$.
The other data is a principal $(\Z_2)^n$-bundle over $Q$ which
provides information on the set of free orbits in $M$.  It turns out
that the orbit space $Q$ together with these two data uniquely
determines a locally standard 2-torus manifold up to equivariant
homeomorphism (Lemma~\ref{form}). When $Q$ is a simple convex
polytope, any principal $(\Z_2)^n$-bundle over it is trivial; so
only a characteristic function matters in this case (\cite{dj}).

\vskip .2cm

The set of isomorphism classes in all principal $(\Z_2)^n$-bundles
over $Q$ can be identified with $H^1(Q;(\Z_2)^n)$. Let $\Lambda(Q)$
be the set of all characteristic functions on $Q$. Then each element
in $H^1(Q;(\Z_2)^n)\times \Lambda(Q)$ determines a locally standard
2-torus manifold with orbit space $Q$.  However, different elements
in the product may produce equivariantly homeomorphic locally
standard 2-torus manifolds. Let $\Aut(Q)$ be the group of
self-homeomorphisms of $Q$ as a manifold with corners.  It naturally
acts on $H^1(Q;(\Z_2)^n)\times \Lambda(Q)$ and one can see that
equivariant homeomorphism classes in locally standard 2-torus
manifolds with orbit space $Q$ can be identified with the coset
$\big( H^1(Q;(\Z_2)^n)\times \Lambda(Q)\big)/\Aut(Q)$, see
Proposition~\ref{QGdiff}.

\vskip .2cm

It is not easy in general to count elements in the coset above, but
we can manage when $Q$ is a compact surface with only one boundary.
In this case, codimension-one faces sit in the boundary circle, so a
characteristic function on $Q$ is nothing but a coloring on a circle
(with vertices) with three colors.

\vskip .2cm

The  paper is organized as follows. In section~\ref{sect:2torus}, we
introduce the notion of locally standard 2-torus manifold and give
several examples. Following Davis and Januszkiewicz \cite{dj}, we
define a characteristic function and construct a locally standard
2-torus manifold from a characteristic function and a principal
bundle in section~\ref{sect:chara}. In section~\ref{sect:euler} we
describe the euler characteristic of a locally standard 2-torus
manifold in terms of its orbit space. Section~\ref{sect:equiv}
discusses three equivalence relations among locally standard 2-torus
manifolds and identify them with some cosets. We count the number of
colorings on a circle in section~\ref{sect:color}. Applying this
result, we find in section~\ref{sect:2dim} the number of equivariant
homeomorphism classes in locally standard 2-torus manifolds when the
orbit space is a compact surface with only one boundary.

\vskip .3cm

\section{2-torus manifolds} \label{sect:2torus}

We denote the quotient additive group $\Z/2\Z$ by $\Z_2$ throughout
this paper.  The natural action of a 2-torus $({\Bbb Z}_2)^n$ of
rank $n$ on ${\Bbb R}^n$ defined by
$$(x_1,...,x_n)\longmapsto ((-1)^{g_1}x_1,...,(-1)^{g_n}x_n),
\qquad (g_1,...,g_n)\in ({\Bbb Z}_2)^n$$ is called {\em the standard
representation} of $({\Bbb Z}_2)^n$. The orbit space is a positive
cone ${\Bbb R}_{\geq 0}^n$. Any real $n$-dimensional faithful
representation of $(\Z_2)^n$ is obtained from the standard
representation by composing a group automorphism of $(\Z_2)^n$, up
to isomorphism.  Therefore the orbit space of the faithful
representation space can also be identified with ${\Bbb R}_{\geq
0}^n$. \vskip .2cm

A 2-torus manifold  $M$ is a closed smooth manifold of dimension $n$
with an effective smooth action of $({\Bbb Z}_2)^n$.  We say that
$M$ is  {\em locally standard} if  for each point $x$ in $M$, there
is a $({\Bbb Z}_2)^n$-invariant neighborhood $V_x$ of $x$ such that
$V_x$ is equivariantly homeomorphic to an invariant open subset of a
real $n$-dimensional faithful representation space of $(\Z_2)^n$.

\begin{rem}
The notion of a torus manifold is introduced in \cite{ha-ma03}. It
is a closed smooth manifold of dimension $2n$ with an effective
smooth action of a compact $n$-dimensional torus $(S^1)^n$ having a
fixed point.  (More precisely speaking, an orientation data on $M$
called an omniorientation in \cite{bp} is incorporated in the
definition.) There is also a notion of local standardness in this
setting (\cite{dj}). Although many interesting examples of torus
manifolds are locally standard (e.g. this is the case for compact
non-singular toric varieties with restricted action of the compact
torus, more generally for torus manifolds with vanishing odd degree
cohomology, \cite{ma-pa03}), the local standardness is not assumed
in the study of \cite{ha-ma03} and \cite{m} because a combinatorial
object called a multi-fan can be defined without assuming it.
However, the existence of a fixed point is not assumed for a 2-torus
manifold unlike a torus manifold.
\end{rem}

For a locally standard 2-torus manifold $M$, the orbit space $Q$ of
$M$ naturally becomes a manifold with corners (see \cite{d2} for the
details of a manifold with corners). Therefore the notion of a face
can be defined for $Q$.  In this paper we assume that a face is
connected.  We call a face of dimension $0$ a {\em vertex}, a face
of dimension one an {\em edge} and a codimension-one face a {\em
facet}.

\vskip .2cm An $n$-dimensional convex polytope $P$ is said to be
{\em simple} if exactly $n$ facets meet at each of its vertices.
Each point of a simple convex polytope $P$ has a neighborhood which
is affine isomorphic to an open subset of the positive cone ${\Bbb
R}_{\geq 0}^n$, so $P$ is an $n$-dimensional manifold with corners.
A locally standard 2-torus manifold $M$ is said to be a {\em small
cover} when its orbit space is a simple convex polytope, see
\cite{dj}.

\vskip .2cm

We call a closed, connected, codimension-one submanifold of $M$ {\em
characteristic} if it is a connected component of the set fixed
pointwise by some ${\Bbb Z}_2$ subgroup.  Since $M$ is compact, $M$
has only finitely many characteristic submanifolds.  The action of
$(\Z_2)^n$ is free outside the union of all characteristic
submanifolds, in other words, a point of $M$ with non-trivial
isotropy subgroup is contained in some characteristic submanifold of
$M$.

\vskip .2cm

Through the quotient map $M\to Q$, a fixed point in $M$ corresponds
to a vertex of $Q$ and a characteristic submanifold of $M$
corresponds to a facet of $Q$. A connected component of the
intersection of $k$ characteristic submanifolds of $M$ corresponds
to a codimension-$k$ face of $Q$, so a codimension-$k$ face of $Q$
is a connected component of the intersection of $k$ facets. In
particular, any codimension-two face of $Q$ is a connected component
of the intersection of two facets of $Q$, which means that $Q$ is
{\em nice}, see \cite{d2}.

\vskip .2cm

We shall give examples of locally standard 2-torus manifolds.

\begin{exam} \label{RPn}
A real projective space ${\Bbb R}P^n$ with the standard $({\Bbb
Z}_2)^n$-action defined by
$$[x_0,x_1, ...,x_n]\longmapsto [x_0, (-1)^{g_1}x_1,...,(-1)^{g_n}x_n],
\qquad (g_1,...,g_n)\in ({\Bbb Z}_2)^n$$ is a locally standard
2-torus manifold. It has $n+1$ isolated points and $n+1$
characteristic submanifolds. The orbit space of ${\Bbb R}P^n$ by
this action is an $n$-simplex, so this locally standard 2-torus
manifold is actually a small cover.
\end{exam}

\begin{exam} \label{char}
Let $S^1$ denote the unit circle in the complex plane $\C$ and
consider two involutions on $S^1\times S^1$ defined by
$$t_1: (z, w)\longmapsto (-z, w),\qquad
t_2: (z,w)\longmapsto ({z}, \bar{w}).$$ Since $t_1$ and $t_2$ are
commutative, they define a $({\Bbb Z}_2)^2$-action on $S^1\times
S^1$, and it is easy to see that $S^1\times S^1$ with this action is
a locally standard 2-torus manifold. It has no fixed point and the
orbit space is ${\Bbb R}P^1\times I=S^1\times I$ where $I$ is a
closed interval.
\end{exam}

\begin{exam} \label{conn}
If $M_1$ and $M_2$ are both locally standard 2-torus manifolds of
the same dimension, then the equivariant connected sum of them along
their free orbits produces a new locally standard 2-torus manifold.
For example, we take ${\Bbb R}P^2$ in Example~\ref{RPn} and
$S^1\times S^1$ in Example~\ref{char} and do the equivariant
connected sum of them along their free orbits. The orbit space of
the resulting locally standard 2-torus manifold $M$ is the connected
sum of a 2-simplex with $S^1\times I$ at their interior points.  $M$
has five characteristic submanifolds and three of them have a fixed
point but the other two have no fixed point.
\end{exam}

If $M$ is a locally standard 2-torus manifold of dimension $n$ and a
subgroup of $(\Z_2)^n$ has an isolated fixed point, then the
isolated point must be fixed by the entire group $(\Z_2)^n$. This
follows from the local standardness of $M$. The following is an
example of a closed $n$-manifold with an effective $(\Z_2)^n$-action
which is not a locally standard 2-torus manifold.

\begin{exam} \label{nloc}
Consider two involutions on the unit sphere $S^2$ of $\R\times \C$
defined by
$$t_1:(x, z)\longmapsto (-x,-z),\qquad
t_2:(x, z)\longmapsto (x,\bar z).$$ Since $t_1$ and $t_2$ are
commutative, they define a $({\Bbb Z}_2)^2$-action on $S^2$.  But
$S^2$ with this action is not a locally standard 2-torus manifold
because the fixed point set of $t_1t_2$ consists of two isolated
points $(0,\pm\sqrt{-1})$ but they are not fixed by the entire group
$(\Z_2)^2$.
\end{exam}

\vskip .3cm

\section{Characteristic functions and principal bundles}
\label{sect:chara}

Let $Q$ be an $n$-dimensional nice manifold with corners. We denote
by ${\mathcal{F}}(Q)$ the set of facets of $Q$. A codimension-$k$
face of $Q$ is a connected component of the intersection of $k$
facets.  We call a map
$$\lambda:{\mathcal{F}}(Q)\longrightarrow ({\Bbb
Z}_2)^n$$ a {\em characteristic function} on $Q$ if it satisfies the
following linearly independent condition:
\begin{quote}
if a codimension-$k$ face $F$ of $Q$ is a connected component of the
intersection of $k$ facets $F_1,\dots,F_k$, then
$\lambda(F_1),\dots, \lambda(F_k)$ are linearly independent when
viewed as vectors of the vector space $(\Z_2)^n$ over the field
$\Z_2$.
\end{quote}
We denote by $G_F$ the subgroup of $(\Z_2)^n$ generated by
$\lambda(F_1),\dots,\lambda(F_k)$.

\begin{rem}
When $n\le 2$, it is easy to see that any $Q$ admits a
characteristic function.  When $n=3$, $Q$ admits a characteristic
function if the boundary of $Q$ is a union of $2$-spheres, which
follows from the Four Color Theorem, but $Q$ may not admit a
characteristic function otherwise. When $n\ge 4$, there is a simple
convex polytope which admits no characteristic function, see
\cite[Nonexamples 1.22]{dj}.
\end{rem}

A characteristic function arises naturally from a locally standard
2-torus manifold $M$ of dimension $n$ with orbit space $Q$. A facet
of $Q$ is the image of a characteristic submanifold of $M$ by the
quotient map $\pi\colon M\to Q$. To each element $F\in
{\mathcal{F}}(Q)$ we assign the nonzero element of $({\Bbb Z}_2)^n$
which fixes pointwise the characteristic submanifold $\pi^{-1}(F)$.
The local standardness of $M$ implies that this assignment satisfies
the linearly independent condition above required for a
characteristic function.

\vskip .2cm

Besides the characteristic function, a principal $(\Z_2)^n$-bundle
over $Q$ will be associated with $M$ as follows.  We take a small
invariant open tubular neighborhood for each characteristic
submanifold of $M$ and remove their union from $M$.  Then the
$(\Z_2)^n$-action on the resulting space is free and its orbit space
can naturally be identified with $Q$, so it gives a principal
$(\Z_2)^n$-bundle over $Q$.

\vskip .2cm

We have associated a characteristic function and a principal
$(\Z_2)^n$-bundle with a locally standard 2-torus manifold.
Conversely, one can reproduce the locally standard 2-torus manifold
from these two data.  This is done by Davis-Januszkiewicz \cite{dj}
when $Q$ is a simple convex polytope, but their construction still
works in our setting. Let $\xi=(E,\kappa, Q)$, where $\kappa\colon
E\to Q$, be a principal $({\Bbb Z}_2)^n$-bundle over $Q$ and let
$\lambda:{\mathcal{F}}(Q)\longrightarrow ({\Bbb Z}_2)^n$ be a
characteristic function on $Q$. We define an equivalence relation
$\sim$ on $E$ as follows: for $u_1, u_2\in E$
$$u_1\sim u_2\Longleftrightarrow \kappa(u_1)=\kappa(u_2) \text{
and $u_1=u_2g$ for some $g\in G_{F}$}$$ where $F$ is the face of $Q$
containing $\kappa(u_1)=\kappa(u_2)$ in its relative interior and
$G_{F}$ is the subgroup of $(\Z_2)^n$ defined at the beginning of
this section. Then the quotient space $E/\sim$, denoted by
$M(\xi,\lambda)$, naturally inherits the $({\Bbb Z}_2)^n$-action
from $E$.

\vskip .2cm

The following is proved in \cite{dj} when $Q$ is a simple convex
polytope, but the same proof works in our setting.

\begin{lem} \label{form}
If a locally standard 2-torus manifold $M$ over $Q$ has $\xi$ as the
associated principal $(\Z_2)^n$-principal bundle and $\lambda$ as
the characteristic function, then there is an equivariant
{\diffeomorphism} from $M(\xi,\lambda)$ to $M$ which covers the
identity on $Q$.
\end{lem}

\vskip .3cm

\section{Euler characteristic of a locally standard 2-torus manifold} \label{sect:euler}

The following formula describes the euler characteristic $\chi(M)$
of a locally standard 2-torus manifold $M$ in terms of its orbit
space.

\begin{thm} \label{euler}
If $M$ is a locally standard 2-torus manifold over $Q$, then
\[
\chi(M)=\sum_F 2^{\dim F}\chi(F,\partial F) =\sum_F 2^{\dim
F}(\chi(F)-\chi(\partial F))
\]
where $F$ runs over all faces of $Q$.
\end{thm}

\begin{proof}
As observed in Section~\ref{sect:chara}, $M$ is the disjoint union
of $2^{\dim F}$ copies of $F\backslash \partial F$ over all faces
$F$ of $Q$. This implies the former identity in the theorem.  The
latter identity is well-known. In fact, it follows from the homology
exact sequence for a pair $(F,\partial F)$.
\end{proof}

When $\dim M=2$, $Q$ is a surface with boundary and each boundary
component is a circle with at least two vertices if it has a vertex.

\begin{cor} \label{2dimeuler}
If $\dim M=2$ and $Q$ has $m$ vertices, then $\chi(M)=4\chi(Q)-m.$
\end{cor}

\begin{proof}
Since $\partial Q$ is a union of circles, $\chi(Q,\partial
Q)=\chi(Q)$. If a boundary circle has no vertex, then it is an edge
without boundary and its euler characteristic is zero.  So we may
neglect it. If $F$ is an edge with a vertex, then it has two
endpoints and $\chi(F,\partial F)=\chi(F)-\chi(\partial F)=-1$, and
if $F$ is a vertex, then $\chi(F,\partial F)=\chi(F)=1$. Since the
number of edges with a vertex and the number of vertices are both
$m$, it follows from Theorem~\ref{euler} that
\[
\chi(M)=2^2\chi(Q)-2m+m=4\chi(Q)-m.
\]
\end{proof}

\begin{rem}
When $\dim M=2$, it is not difficult to see that $M$ is orientable
if and only if $Q$ is orientable and the characteristic function
$\lambda \colon \mathcal F(Q) \to (\Z_2)^2$ associated with $M$
assigns exactly two elements to each boundary component of $Q$ with
a vertex, cf. \cite{na-ni05}. Therefore one can find the
{\diffeomorphism} type of $M$ from the corollary above and the
characteristic function $\lambda$.
\end{rem}

\vskip .3cm

\section{Classification of locally standard 2-torus manifolds} \label{sect:equiv}

In this section we introduce three notions of equivalence in locally
standard 2-torus manifolds over $Q$ and identify each set of
equivalence classes with a coset of $H^1(Q;(\Z_2)^n)\times
\Lambda(Q)$ by some action.

Following Davis and Januszkiewicz \cite{dj} we say that two locally
standard 2-torus manifolds $M$ and $M'$ over $Q$ are {\em
equivalent} if there is a {\diffeomorphism} $f: M\longrightarrow M'$
together with an element $\sigma\in \GL(n,{\Bbb Z}_2)$ such that
\begin{enumerate}
\item[(1)]
$f(gx)=\sigma(g)f(x)$ for all $g\in (\Z_2)^n$ and $x\in M$, and
\item[(2)]
$f$ induces the identity on the orbit space $Q$.
\end{enumerate}
When we classify locally standard 2-torus manifolds up to the above
equivalence, it suffices to consider locally standard 2-torus
manifolds of the form $M(\xi,\lambda)$ by Lemma~\ref{form}. We
denote by $\xi^\sigma$ the principal $(\Z_2)^n$-bundle $\xi$ with
$(\Z_2)^n$-action through $\sigma\in \GL(n,\Z_2)$. Then it would be
obvious that $M(\xi',\lambda')$ is equivalent to $M(\xi,\lambda)$ if
and only if there exists $\sigma\in \GL(n,{\Bbb Z}_2)$ such that
$\xi'=\xi^\sigma$ and $\lambda'=\sigma\circ\lambda$.

\vskip .2cm We denote by $\mathcal{P}(Q)$ the set of all principal
$(\Z_2)^n$-bundles over $Q$.  Since the classifying space of
$(\Z_2)^n$ is an Eilenberg-Maclane space $K((\Z_2)^n,1)$,
$\mathcal{P}(Q)$ can naturally be identified with $H^1(Q;(\Z_2)^n)$
and the action of $\sigma$ sending $\xi$ to $\xi^\sigma$ is nothing
but the action on $H^1(Q;(\Z_2)^n)$ induced from the automorphism
$\sigma$ on the coefficient $(\Z_2)^n$. With this understood, the
above fact implies the following.

\begin{prop} \label{equicor}
The set of equivalence classes in locally standard 2-torus manifolds
over $Q$ bijectively corresponds to the coset
$$\GL(n,\Z_2)\backslash\big(H^1(Q;(\Z_2)^n)\times \Lambda(Q)\big)$$
by the diagonal action.
\end{prop}

The action of $\GL(n,\Z_2)$ on $H^1(Q;(\Z_2)^n)\times \Lambda(Q)$ is
free when $Q$ has a vertex by the following lemma.

\begin{lem} \label{free}
If $Q$ has a vertex, then the  action of $\GL(n,{\Bbb Z}_2)$ on
$\Lambda(Q)$ is free and
$|\Lambda(Q)|=|\GL(n,\Z_2)\backslash\Lambda(Q)|\prod_{k=1}^n(2^n-2^{k-1})$.
\end{lem}

\begin{proof}
Suppose that $\lambda=\sigma\circ\lambda$ for some
$\lambda\in\Lambda(Q)$ and $\sigma\in\GL(n,\Z_2)$.  Take a vertex of
$Q$ and let $F_1,\dots,F_n$ be the facets of $Q$ meeting at the
vertex.  Then
$$(\lambda(F_1),\dots,\lambda(F_n))=\sigma(\lambda(F_1),\dots,\lambda(F_n)).$$
Since the matrix $(\lambda(F_1),\dots,\lambda(F_n))$ is
non-singular, $\sigma$ is the identity matrix.  This proves the
former statement in the lemma. Then the latter statement follows
from the well-known fact that
$|\GL(n,\Z_2)|=\prod_{k=1}^n(2^n-2^{k-1})$, see \cite{al-be95}.
\end{proof}

Lemma~\ref{free} is also helpful to count the number of elements in
$\Lambda(Q)$.  Here is an example.

\begin{exam} (The number of characteristic functions on a prism.)
There exist seven combinatorially inequivalent 3-polytopes with six
vertices (see \cite[Theorem 6.7]{e}) and only one of them is simple,
which is a prism $P^3$.

Let us count the number of characteristic functions on $P^3$. $P^3$
has five facets, consisting of three  square facets and two
triangular facets.  We denote the three square facets by $F_1, F_2,
F_4$, and the two triangular facets by $F_3,F_5$. The facets
$F_1,F_2,F_3$ intersect at a vertex and we may assume that a
characteristic function $\lambda$ on $P^3$ takes the standard basis
$\text{\bf e}_1, \text{\bf e}_2, \text{\bf e}_3$ of $(\Z_2)^3$ on
$F_1, F_2, F_3$ respectively through the action of $\GL(3,\Z_2)$ on
$(\Z_2)^3$. The characteristic function $\lambda$ must satisfy the
linearly independent condition at each vertex of $P^3$.  This
requires that the values of $\lambda$ on the remaining facets $F_4,
F_5$ must be as follows:
\[
(\lambda(F_4),\lambda(F_5))=(\text{\bf e}_1+\text{\bf e}_2+\text{\bf
e}_3,\text{\bf e}_3) \quad\text{or}\quad (\text{\bf e}_1+\text{\bf
e}_2, a\text{\bf e}_1+b\text{\bf e}_2+\text{\bf e}_3) \] where
$a,b\in\Z_2$. Therefore
$$|\GL(3,{\Bbb Z}_2)\backslash\Lambda(P^3)|=5\quad\text{and}\quad
\vert \Lambda(P^3)\vert=5\vert\GL(3,{\Bbb Z}_2)\vert=840$$ by
Lemma~\ref{free}.
\end{exam}

\vskip .2cm

Another natural equivalence relation among locally standard 2-torus
manifolds is equivariant {\diffeomorphism}. An {\em automorphism} of
$Q$ is a self-{\diffeomorphism} of $Q$ as a manifold with corners,
and we denote the group of automorphisms of $Q$ by $\Aut(Q)$.
Similarly, an {\em automorphism} of ${\mathcal{F}}(Q)$ is a
bijection from ${\mathcal{F}}(Q)$ to itself which preserves the
poset structure of ${\mathcal{F}}(Q)$ defined by inclusions of
faces, and we denote the group of automorphisms of
${\mathcal{F}}(Q)$ by $\Aut({\mathcal{F}}(Q))$. An automorphism of
$Q$ induces an automorphism of ${\mathcal{F}}(Q)$, so we have a
natural homomorphism
\begin{equation} \label{Phi}
\Phi\colon \Aut(Q)\to \Aut({\mathcal{F}}(Q)).
\end{equation}
We note that $\Aut(\mathcal{F}(Q))$ acts on $\Lambda(Q)$ by sending
$\lambda\in \Lambda(Q)$ to $\lambda\circ h$ for $h\in
\Aut(\mathcal{F}(Q))$.

\vskip .2cm

\begin{lem} \label{Gdiff}
$M(\xi,\lambda)$ is equivariantly {\diffeomorphic} to
$M(\xi',\lambda')$ if and only if there is an $h\in \Aut(Q)$ such
that $\lambda'=\lambda\circ \Phi(h)$ and $h^*(\xi')=\xi$ in
$\mathcal{P}(Q)$, where $h^*(\xi')$ denotes the bundle induced from
$\xi'$ by $h$.
\end{lem}

\begin{proof} If $M(\xi,\lambda)$ is equivariantly {\diffeomorphic} to
$M(\xi',\lambda')$, then there is an equivariant \diffeomorphism\
$H\colon M(\xi', \lambda')\to M(\xi,\lambda)$ and it is easy to see
that the automorphism of $Q$ induced from $H$ is the desired $h$ in
the theorem.

Conversely, suppose that there is an $h\in \Lambda(Q)$ such that
$\lambda'=\lambda\circ \Phi(h)$ and $\xi'=h^*(\xi)$ in
$\mathcal{P}(Q)$. Then there is a bundle map $\hat h\colon \xi'\to
\xi$ which covers $h$, and $\hat h$ descends to a map $H$ from
$M(\xi',\lambda')$ to $M(\xi,\lambda)$ because
$\lambda'=\lambda\circ \Phi(h)$.  It is not difficult to see that
$H$ is an equivariant {\diffeomorphism}.
\end{proof}

$\Aut(Q)$ naturally acts on $H^1(Q;(\Z_2)^n)$ and the canonical
bijection between $\mathcal{P}(Q)$ and $H^1(Q;(\Z_2)^n)$ is
equivariant with respect to the actions of $\Aut(Q)$.

\begin{prop} \label{QGdiff}
The set of equivariant {\diffeomorphism} classes in all locally
standard 2-torus manifolds over $Q$ bijectively corresponds to the
coset
$$(H^1(Q,(\Z_2)^n)\times \Lambda(Q))/\Aut(Q)$$
by the diagonal action of $\Aut(Q)$. If $Q$ is a simple convex
polytope, then the set of equivariant {\diffeomorphism} classes in
all small covers over $Q$ bijectively corresponds to the coset
$\Lambda(Q)/\Aut(\mathcal{F}(Q))$.
\end{prop}

\begin{proof}
The former statement in the proposition follows from
Lemma~\ref{Gdiff}. If $Q$ is a simple polytope, then
$H^1(Q;(\Z_2)^n)=0$.  Therefore, the latter statement in the
proposition follows if we prove that the map $\Phi$ in (\ref{Phi})
is surjective when $Q$ is a simple convex polytope.

A simple polytope $Q$ has a simplicial polytope $Q^*$ as its dual
and the face poset ${\mathcal F}(Q)$ is same as ${\mathcal F}(Q^*)$
with reversed inclusion relation. Therefore $\Aut({\mathcal
F}(Q))=\Aut({\mathcal F}(Q^*))$. Since $Q^*$ is simplicial, an
element $\varphi$ of $\Aut({\mathcal F}(Q^*))$ is realized by a
simplicial automorphism on the boundary of $Q^*$, so it extends to
an automorphism of $Q^*$.  Since $Q$ is dual to $Q^*$, the
automorphism of $Q^*$ determines a bijection on the vertex set of
$Q$ and hence an automorphism of $Q$ which induces the chosen
$\varphi$.
\end{proof}


Our last equivalence relation is a combination of the previous two
relations. We say that two locally standard 2-torus manifolds $M$
and $M'$ over $Q$ are {\em weakly equivariantly \diffeomorphic} if
there is a {\diffeomorphism} $f\colon M\to M'$ together with
$\sigma\in \GL(n,\Z_2)$ such that $f(gx)=\sigma(g)f(x)$ for any
$g\in (\Z_2)^n$ and $x\in M$.  We note that $f$ induces an
automorphism of $Q$ but it may not be the identity on $Q$. The
observation above shows that $M(\xi,\lambda)$ and $M(\xi',\lambda')$
are weakly equivariantly {\diffeomorphic} if and only if there are
$h\in \Aut(Q)$ and $\sigma\in \GL(n,\Z_2)$ such that
$\xi'=h^*(\xi^\sigma)$ and $\lambda'=\sigma\circ\lambda\circ h$.  It
follows that

\begin{prop}
The set of weakly equivariant {\diffeomorphism} classes in locally
standard 2-torus manifolds over $Q$ bijectively corresponds to the
double coset
\[
\GL(n,\Z_2)\backslash\big(H^1(Q;(\Z_2)^n)\times
\Lambda(Q)\big)/\Aut(Q)
\]
by the diagonal actions of $\Aut(Q)$ and $\GL(n,\Z_2)$. If $Q$ is a
simple convex polytope, then the set of weakly equivariant
{\diffeomorphism} classes in small covers over $Q$ bijectively
corresponds to the double coset
\begin{equation} \label{doub}
\GL(n,\Z_2)\backslash\Lambda(Q)/\Aut(\mathcal{F}(Q)).
\end{equation}
\end{prop}

\begin{rem} When $Q$ is a right-angled regular
hyperbolic polytope (such $Q$ is the dodecahedron, the 120-cell or
an $m$-gon with $m\ge 5$), it is shown in \cite[Theorem
3.3]{ga-sc02} that the double coset (\ref{doub}) agrees with the set
of hyperbolic structures in small covers over $Q$.  This together
with Mostow rigidity implies that when $\dim Q\ge 3$, that is, when
$Q$ is the dodecahedron or the 120-cell, the double coset
(\ref{doub}) agrees with the set of homeomorphism classes in small
covers over $Q$ (\cite[Corollary 3.4]{ga-sc02}), i.e., the natural
surjective map from the double coset to the set of homeomorphism
classes in small covers over $Q$ is bijective for such $Q$. However,
this last statement does not hold for an $m$-gon $Q$ with $m\ge 6$
although it holds for $m=3,4,5$, see the remark following
Example~\ref{Cm} in the next section.
\end{rem}

\vskip .3cm

\section{Enumeration of colorings on a circle} \label{sect:color}

When $\dim Q=2$, each boundary component is a circle with at least
two vertices if it has a vertex, and any two non-zero elements in
$({\Bbb Z}_2)^2$ form a basis of $({\Bbb Z}_2)^2$; so a
characteristic function on $Q$ is equivalent to coloring arcs on the
boundary circles with three colors in such a way that any two
adjacent arcs have different colors.

\vskip .2cm

Let $\SA(m)$ be a circle with $m$ $(\ge 2)$ vertices. A coloring on
$\SA(m)$ (with three colors) means to color arcs of $\SA(m)$ in such
a way that any adjacent arcs have different colors. We denote by
$\CO(m)$ the set of all colorings on $\SA(m)$ and set
\[
A(m):=|\CO(m)|.
\]

\begin{lem} \label{A(m)}
$A(m)=2^m+(-1)^m2.$
\end{lem}

\begin{proof}
Let $L(m)$ be a segment with $m+1$ vertices including the endpoints,
so $L(m)$ has $m$ segments. The number of coloring segments of
$L(m)$ with three colors in such a way that any adjacent segments
have different colors is $3\cdot 2^{m-1}$. If the two end segments
have different colors, then it produces a coloring on $\SA(m)$ by
gluing the end points of $L(m)$. If the two end segments have the
same color, then it produces a coloring on $\SA(m-1)$ by gluing the
end segments of $L(m)$. Thus, we have that
\begin{equation} \label{Am1}
A(m)+A(m-1)=3\cdot 2^{m-1}.
\end{equation}
It follows that
\[
A(m)-2A(m-1)=-(A(m-1)-2A(m-2))=\cdots=(-1)^{m-3}(A(3)-2A(2))
\]
and a simple observation shows that $A(3)=A(2)=6$, so
\begin{equation} \label{Am2}
A(m)-2A(m-1)=(-1)^m6.
\end{equation}
The lemma then follows from (\ref{Am1}) and (\ref{Am2}).
\end{proof}

\bigskip

We think of $\P(m)$ as the unit circle of $\C$ with $m$ vertices
$e^{2\pi k/m}$ $(k=0,1,\dots,m-1)$.  Let ${\frak D}_m$ be the
dihedral group of order $2m$ consisting of $m$ rotations of $\C$ by
angles $2\pi k/m$ $(k=0,1,\dots,m-1)$ and $m$ reflections with
respect to lines in $\C$ obtained by rotating the real axis by
angles $\pi k/m$ $(k=0,1,\dots,m-1)$. Then the action of ${\frak
D}_m$ on $\P(m)$ preserves the vertices so that ${\frak D}_m$ acts
on the set $\CO(m)$. With this understood we have

\vskip .2cm

\begin{thm} \label{B(m)}
Let $\varphi$ denote the Euler's totient function, that is,
$\varphi(1)=1$ and $\varphi(N)$ for a positive integer $N(\geq 2)$
is the number of positive integers both less than $N$ and coprime to
$N$. Then
\[
\vert\CO(m)/{\frak D}_m\vert=\frac{1}{2m}\Big(\sum_{2\leq
d|m}\varphi(m/d)A(d)+\frac{1+(-1)^m}{2}\cdot 3\cdot 2^{m/2}\cdot
\frac{m}{2}\Big).
\]
\end{thm}

\begin{proof}
The famous Burnside Lemma or Cauchy-Frobenius Lemma (see
\cite{al-be95}) says that if $G$ is a finite group and $X$ is a
finite $G$-set, then
\[
|X/G|=\frac{1}{|G|}\sum_{g\in G}|X^g|
\]
where $X^g$ denotes the set of $g$-fixed points in $X$. We apply
this formula to our ${\frak D}_m$-set $\Lambda(m)$.  Let $a\in
{\frak D}_m$ be the rotation by angle $2\pi/m$ and $b\in {\frak
D}_m$ be the reflection with respect to the real axis. Then we have
\begin{equation} \label{eq1}
\vert\CO(m)/{\frak
D}_m\vert=\frac{1}{2m}\sum_{k=0}^{m-1}\big(|\CO(m)^{a^k}|
+|\CO(m)^{a^k b}|\big).
\end{equation}
Here, if $d$ is the greatest common divisor of $k$ and $m$, then
$\CO(m)^{a^k}=\CO(m)^{a^d}$ because the subgroup generated by $a^k$
is same as that by $a^d$. Since $\CO(m)^{a^d}=\CO(d)$ and $\CO(1)$
is empty, we have
\begin{equation} \label{eq2}
\sum_{k=0}^{m-1}|\CO(m)^{a^k}|=\sum_{2\leq d\mid m}\varphi(m/d)A(d).
\end{equation}
On the other hand, since $a^kb$ is a reflection with respect to the
line in $\C$ obtained by rotating the real axis by angle $\pi k/m$,
we have
\begin{equation} \label{eq3}
|\CO(m)^{a^kb}|=\begin{cases} 3\cdot 2^{{m/2}}
\quad&\text{when $m$ is even and $k$ is odd,}\\
0 \quad&\text{otherwise.}
\end{cases}
\end{equation}
Putting (\ref{eq2}) and (\ref{eq3}) into (\ref{eq1}), we obtain the
formula in the theorem.
\end{proof}

\begin{exam}
As is well known, $\varphi(p^n)=p^{n-1}(p-1)$ for any prime number
$p$ and positive integer $n$, and $\varphi(ab)=\varphi(a)\varphi(b)$
for relatively prime positive integers $a$ and $b$. We set
$$B(m):=\vert\CO(m)/{\frak D}_m\vert.$$
Using the formula in Theorem~\ref{B(m)} together with
Lemma~\ref{A(m)}, one finds that
\[
\begin{split}
&B(2)=3, \quad B(3)=1, \quad B(4)=6, \quad B(5)=3, \quad B(6)=13,\\
&B(7)=9, \quad B(8)=30, \quad B(9)=29, \quad B(10)=78,\\
&B(2^k)=2^{2^k-k-1}+3\cdot 2^{2^{k-1}-2}+\sum_{i=1}^k2^{2^{i-1}-i-1}\\
&B(p^k)=\sum_{i=1}^k\frac{1}{2p^i}(2^{p^i}-2^{p^{i-1}})\\
&B(2p)=\frac{1}{4p}(4^p+(3p+1)2^p+6p-6)
\\
&B(pq)=\frac{1}{2pq}(2^{pq}-2^p-2^q+2)+\frac{1}{2p}(2^p-2)+\frac{1}{2q}
(2^q-2)
\end{split}
\]
where $p$ is an odd prime number and $q$ is another odd prime
number.
\end{exam}

\begin{rem}
The same argument as above works for coloring $\P(m)$ with $s$
colors. In this case the identity in Lemma~\ref{A(m)} turns into
\[
A_s(m)=(s-1)^m+(-1)^m(s-1)
\]
and if we denote by $\Lambda_s(m)$ the set of all coloring on $S(m)$
with $s$ colors, then the formula in Theorem~\ref{B(m)} turns into
\[
|\Lambda_s(m)/{\frak D}_m|= \frac{1}{2m}\Big(\sum_{2\leq
d|m}\varphi(m/d)A_s(d)+ \frac{1+(-1)^m}{2}\cdot s\cdot
(s-1)^{m/2}\cdot \frac{m}{2}\Big).
\]
\end{rem}

\vskip .2cm

The computation of $|\GL(2,\Z_2)\backslash \Lambda(m)/{\frak D}_m|$
can be done in a similar fashion to the above but is rather
complicated. We note that the action of $\GL(2,\Z_2)$ on
$\Lambda(m)$ is permutation of the 3 colors used to color $S(m)$.
$\GL(2,\Z_2)$ consists of 6 elements and three of them are of order
2 and two of them is of order 3.

\begin{thm} \label{C(m)}
Let $\alpha$ and $\beta$ be the functions defined as follows:
\[
\begin{split}
&\alpha(1)=1,\ \alpha(2)=3,\ \alpha(3)=2,\ \alpha(6)=4,\\
&\beta(1)=0,\ \beta(2)=2,\ \beta(3)=2,\ \beta(6)=4.
\end{split}
\]
Then $|\GL(2,\Z_2)\backslash \Lambda(m)/{\frak D}_m|$ is given by
\[
\frac{1}{2m}\Big[\sum_{d|m}\Big\{\varphi(m/d)\cdot\frac{1}{6}\Big(
\alpha\big((m/d,6)\big)A(d)+\beta\big((m/d,6)\big)A(d-1)\Big)\Big\}+E(m)\Big]
\]
where $(m/d,6)$ denotes the greatest common divisor of $m/d$ and $6$,
$A(q)=2^q+(-1)^q2$ as before, and
\[
E(m)=\begin{cases} \frac{m}{6}A(\frac{m+1}{2}) \quad&\text{if $m$ is odd},\\
m\cdot 2^{m/2-1}\quad&\text{if $m$ is even.}
\end{cases}
\]
\end{thm}

\begin{proof}
Applying the Burnside Lemma to our ${\frak D}_m$-set $\Gamma(m):=
\GL(2,\Z_2)\backslash \Lambda(m)$, we have
\begin{equation}\label{dcoset}
\begin{split}
&|\GL(2,\Z_2)\backslash \Lambda(m)/{\frak D}_m|
=\frac{1}{2m}\sum_{g\in {\frak D}_m}|\Gamma(m)^g|\\
=&\frac{1}{2m}\sum_{k=0}^{m-1}\big(|\Gamma(m)^{a^k}|+|\Gamma(m)^{a^kb}|\big)
=\frac{1}{2m}\Big[\sum_{d|m}\varphi(m/d)|\Gamma(m)^{a^d}|+
\sum_{k=0}^{m-1}|\Gamma(m)^{a^kb}|\Big].
\end{split}
\end{equation}
We need to analyze $|\Gamma(m)^{a^d}|$ with $d|m$ and
$|\Gamma(m)^{a^kb}|$.

First we shall treat $|\Gamma(m)^{a^d}|$ with $d|m$. Note that
$\lambda\in \Lambda(m)$ is a representative of $\Gamma(m)^{a^d}$ if
and only if there is $\sigma\in\GL(2,\Z_2)$ such that
\begin{equation}\label{lambda}
\sigma\circ\lambda=\lambda\circ a^d.
\end{equation}
Since $a^d$ is of order $m/d$, the repeated use of (\ref{lambda})
shows that
\begin{equation}\label{order}
\sigma^{m/d}=1.
\end{equation}
The identity (\ref{lambda}) implies that the $\lambda$ satisfying
(\ref{lambda}) can be determined by the coloring restricted to the
union of a consecutive $d$ arcs, say $T$, and it also tells us how
to recover $\lambda$ from the coloring on $T$.

Let $\mu$ be a coloring on $T$. We shall count colorings $\lambda$
on $S(m)$ which are extensions of $\mu$ and satisfy (\ref{lambda})
for some $\sigma\in\GL(2,\Z_2)$. To each $\sigma$ satisfying
(\ref{order}), there is a unique extension to $S(m)$ which satisfies
(\ref{lambda}).  However, the extended one may not be a coloring,
i.e., two arcs meeting at a junction of $T$ and its translations by
rotations $(a^d)^r$ $(r=1,\dots,m/d-1)$ may have the same color. Let
$t$ and $t'$ be the end arcs of $T$ such that the rotation of $t$ by
$a^{d-1}$ is $t'$. (Note: When $d=1$, we understand $t=t'$ and then
the subsequent argument works.) The extended one is a coloring if
and only if
\begin{equation} \label{sigma}
\sigma(\mu(t))\not=\mu(t').
\end{equation}
As is easily checked, the number of $\sigma$ satisfying conditions
(\ref{order}) and (\ref{sigma}) is $\alpha((m/d,6))$ if
$\mu(t)\not=\mu(t')$ and is $\beta((m/d,6))$ if $\mu(t)=\mu(t')$. On
the other hand, the number of $\mu$ with $\mu(t)\not=\mu(t')$ is
$A(d)$ and that with $\mu(t)=\mu(t')$ is $A(d-1)$. It follows that
the number of $\lambda$ satisfying (\ref{lambda}) for some $\sigma$
is $\alpha((m/d,6))A(d)+\beta((m/6,d))A(d-1)$. This proves that
\begin{equation}\label{gamma}
|\Gamma(m)^{a^d}|=\frac{1}{6}\Big(\alpha\big((m/d,6)\big)A(d)+
\beta\big((m/6,d)\big)A(d-1)\Big)
\end{equation}
since the action of $\GL(2,\Z_2)$ on $\Lambda(m)$ is free by
Lemma~\ref{free} and the order of $\GL(2,\Z_2)$ is 6.

Next we shall treat $|\Gamma(m)^{a^kb}|$. The argument is similar to
the above. As before, $\lambda\in \Lambda(m)$ is a representative of
$\Gamma(m)^{a^kb}$ if and only if there is $\sigma\in\GL(2,\Z_2)$
such that
\begin{equation}\label{lambda2}
\sigma\circ\lambda=\lambda\circ a^kb.
\end{equation}
Since $a^kb$ is of order two, the repeated use of (\ref{lambda2})
shows that
\begin{equation}\label{order2}
\sigma^2=1.
\end{equation}

Suppose that $m$ is odd. Then the line fixed by $a^kb$ goes through
a vertex, say $v$, of $S(m)$ and the midpoint of the arc, say $e'$,
of $S(m)$ opposite to the vertex $v$.  Let $H$ be the union of
$(m+1)/2$ consecutive arcs starting from $v$ and ending at $e'$. Let
$e$ be the other end arc of $H$ different from $e'$. The arc $e$ has
$v$ as a vertex. Let $\nu$ be a coloring on $H$ and let
$\sigma\in\GL(2,\Z_2)$ satisfy (\ref{order2}). Then $\nu$ has an
extension to a coloring of $S(m)$ satisfying (\ref{lambda2}) if and
only if
\begin{equation*} \label{beta}
\sigma(\nu(e))\not=\nu(e)
\quad\text{and}\quad\sigma(\nu(e'))=\nu(e').
\end{equation*}
It follows that $\nu(e)$ must be different from $\nu(e')$ and there
is only one $\sigma$ satisfying the two identities above for each
such $\nu$.  Since the number of $\nu$ with $\nu(e)\not=\nu(e')$ is
$A((m+1)/2)$, so is the number of $\lambda\in\Lambda(m)$ satisfying
(\ref{lambda2}) for some $\sigma$. It follows that
$|\Gamma(m)^{a^kb}|=\frac{1}{6}A((m+1)/2)$ and hence
\begin{equation}\label{akbmod}
\sum_{k=0}^{m-1}|\Gamma(m)^{a^kb}|=\frac{m}{6}A((m+1)/2).
\end{equation}

Suppose that $m$ is even and $k$ is odd. Then the line fixed by
$a^kb$ goes through the midpoints of two opposite arcs, say $e$ and
$e'$, of $S(m)$. Let $H$ be the union of consecutive $m/2+1$ arcs
starting from $e$ and ending at $e'$. Let $\nu$ be a coloring on $H$
and let $\sigma\in \GL(2,\Z_2)$ satisfy (\ref{order2}).  Then $\nu$
has an extension to a coloring of $S(m)$ satisfying (\ref{lambda2})
if and only if
$$\sigma(\nu(e))=\nu(e)\quad\text{and}\quad \sigma(\nu(e'))=\nu(e').$$
If $\nu(e)\not=\nu(e')$ then such $\sigma$ must be the identity, and
if $\nu(e)=\nu(e')$ then there are two such $\sigma$ one of which is
the identity. Since the number of $\nu$ with $\nu(e)\not=\nu(e')$ is
$A(m/2+1)$ and that with $\nu(e)=\nu(e')$ is $A(m/2)$, the number of
$\lambda\in \Lambda(m)$ satisfying (\ref{lambda2}) for some $\sigma$
is $A(m/2+1)+2A(m/2)$. It follows that
\begin{equation}\label{akbod}
\sum_{k=0,k:\text{odd}}^{m-1}|\Gamma(m)^{a^kb}|=\frac{m}{12}
\big(A(m/2+1)+2A(m/2)\big).
\end{equation}

Suppose that $m$ is even and $k$ is even. Then the line fixed by
$a^kb$ goes through two opposite vertices, say $v$ and $v'$, of
$S(m)$. Let $H$ be the union of consecutive $m/2$ arcs starting from
$v$ and ending at $v'$. Let $e$ and $e'$ be the end arcs of $H$
which respectively have $v$ and $v'$ as a vertex. Let $\nu$ be a
coloring on $H$ and let $\sigma\in \GL(2,\Z_2)$ satisfy
(\ref{order2}). Then $\nu$ has an extension to a coloring of $S(m)$
satisfying (\ref{lambda2}) if and only if
$$\sigma(\nu(e))\not=\nu(e)\quad\text{and}\quad
\sigma(\nu(e'))\not=\nu(e').$$ If $\nu(e)\not=\nu(e')$ then there is
only one such $\sigma$, and if $\nu(e)=\nu(e')$ then there are two
such $\sigma$. Since the number of $\nu$ with $\nu(e)\not=\nu(e')$
is $A(m/2)$ and that with $\nu(e)=\nu(e')$ is $A(m/2-1)$, the number
of $\lambda\in \Lambda(m)$ satisfying (\ref{lambda2}) for some
$\sigma$ is $A(m/2)+2A(m/2-1)$. It follows that
\begin{equation}\label{akbev}
\sum_{k=0, k:\text{even}}^{m-1}|\Gamma(m)^{a^kb}|=\frac{m}{12}
\big(A(m/2)+2A(m/2-1)\big).
\end{equation}

Thus, when $m$ is even, it follows from (\ref{akbod}) and
(\ref{akbev}) that
\begin{equation}\label{akbmev}
\sum_{k=0}^{m-1}|\Gamma(m)^{a^kb}|=\frac{m}{12}
\big(A(m/2+1)+3A(m/2)+2A(m/2-1)\big)=m\cdot 2^{m/2-1}
\end{equation}
where we used $A(q)=2^q+(-1)^q2$ at the latter identity.

The theorem now follows from (\ref{dcoset}), (\ref{gamma}),
(\ref{akbmod}) and (\ref{akbmev}).
\end{proof}

\begin{rem} When $m$ is even,
$\Lambda(m)$ contains exactly three colorings with two colors and it
defines the unique element in the double coset
$\GL(2,\Z_2)\backslash \Lambda(m)/{\frak D}_m$.
\end{rem}

\begin{exam} \label{Cm}
We set
\[
C(m):=|\GL(2,\Z_2)\backslash \Lambda(m)/{\frak D}_m|.
\]
Using the formula in Theorem~\ref{C(m)}, one finds that
\[
\begin{split}
&C(2)=1, \quad C(3)=1, \quad C(4)=2, \quad C(5)=1, \quad C(6)=4,\quad C(7)=3\\
&C(8)=8, \quad C(9)=8, \quad C(10)=18.\quad C(11)=21, \quad C(12)=48.\\
\end{split}
\]
\end{exam}

\vskip .2cm

We conclude this section with a remark. When $Q$ is an $m$-gon
$(m\ge 3)$, a small cover over $Q$ is a closed surface with euler
characteristic $4-m$ and the cardinality of the set of homeomorphism
classes in small covers over $Q$ is one (resp. two) when $m$ is odd
(resp. even). On the other hand, the double coset (\ref{doub})
agrees with $\GL(2,\Z_2)\backslash \Lambda(m)/{\frak D}_m$ and we
see from Theorem~\ref{C(m)} that its cardinality is strictly larger
than 2 when $m\ge 6$. So, the natural surjective map from the double
coset (\ref{doub}) to the set of homeomorphism classes in small
covers over $Q$ is not injective when $Q$ is an $m$-gon with $m\ge
6$. However, it is bijective when $m=3,4,5$, see Example~\ref{Cm}.

\vskip .3cm

\section{Locally standard 2-torus manifolds of dimension two} \label{sect:2dim}

We shall enumerate the number of equivariant {\diffeomorphism}
classes in locally standard 2-torus manifolds with orbit space $Q$
when $Q$ is a compact surface with only one boundary.

\begin{thm}
Suppose that $Q$ is a compact surface with only one boundary
component with $m$ $(\ge 2)$ vertices and set
\[
h(Q):=|H^1(Q;(\Z_2)^2)/\Aut(Q)|.
\]
Then the number of equivariant {\diffeomorphism} classes in locally
standard 2-torus manifolds over $Q$ is $h(Q)B(m)$, where
$B(m)=|\Lambda(m)/{\frak D}_m|$ is the number discussed in the
previous section.
\end{thm}

\begin{proof}
By Corollary~\ref{QGdiff} it suffices to count the number of orbits
in $H^1(Q;(\Z_2)^2)\times \Lambda(Q)$ under the diagonal action of
$\Aut(Q)$. Since $Q$ has only one boundary component and $m$
vertices, $\Lambda(Q)$ can be identified with $\Lambda(m)$ in
Section~\ref{sect:color} and $\Aut({\mathcal F}(Q))$ is isomorphic
to the dihedral group ${\frak D}_m$.

\vskip .2cm

Let $H$ be the normal subgroup of $\Aut(Q)$ which acts on
$H^1(Q;(\Z_2)^2)$ trivially. We claim that the restriction of the
natural homomorphism
\begin{equation} \label{surj}
\Aut(Q)\to \Aut({\mathcal F}(Q))\cong {\frak D}_m
\end{equation}
to $H$ is still surjective. An automorphism of $Q$ (as a manifold
with corners) which rotates the boundary circle and fixes the
exterior of its neighborhood is an element of $H$. Therefore $H$
contains all rotations in ${\frak D}_m$. It is not difficult to see
that any closed surface admits an involution which has
one-dimensional fixed point component and acts trivially on the
cohomology with $\Z_2$ coefficient. Since $Q$ is obtained from a
closed surface by removing an invariant open disk centered at a
point in the one-dimensional fixed point set, $Q$ admits an
involution which reflects the boundary circle and lies in $H$. This
implies the claim.

\vskip .2cm

Let $K$ be the kernel of the homomorphism $\Aut(Q)\to \Aut({\mathcal
F}(Q))$.  Then
\begin{equation} \label{set}
|\big(H^1(Q;(\Z_2)^2)\times \Lambda(Q)\big)/\Aut(Q)|=
|\big(H^1(Q;(\Z_2)^2)/K\times \Lambda(Q)\big)/\Aut(Q)|.
\end{equation}
For any element $g$ in $\Aut(Q)$, there is an element $h$ in $H$
such that $gh$ lies in $K$ because the map (\ref{surj}) restricted
to $H$ is surjective. Since $H$ acts trivially on $H^1(Q;(\Z_2)^2)$,
this shows that an $\Aut(Q)$-orbit in $H^1(Q;(\Z_2)^2)$ is same as
an $K$-orbit.  This means that the induced action of $\Aut(Q)$ on
$H^1(Q;(\Z_2)^2)/K$ is trivial. Therefore the right hand side at
(\ref{set}) reduces to
\[
|H^1(Q;(\Z_2)^2)/\Aut(Q)\big)|\ |\Lambda(Q)/\Aut(Q)|.
\]
Here the first factor is $h(Q)$ by definition and the second one
agrees with $|\Lambda(m)/{\frak D}_m|=B(m)$ because of the
surjectivity of the map (\ref{surj}), proving the theorem.
\end{proof}

\begin{exam}
$H^1(Q;(\Z_2)^2)$ is isomorphic to $H^1(Q;\Z_2)\oplus H^1(Q;\Z_2)$
and the action of $\Aut(Q)$ on the direct sum is diagonal. When $Q$
is a disk, $h(Q)=1$. When $Q$ is a real projective plane with an
open disk removed, $H^1(Q;\Z_2)$ is isomorphic to $\Z_2$ and the
action of $\Aut(Q)$ on it is trivial.  Therefore, $h(Q)=4$ in this
case.  When $Q$ is a torus with an open disk removed, $H^1(Q;\Z_2)$
is isomorphic to $(\Z_2)^2$.  The action of $\Aut(Q)$ on it is
non-trivial and it is not difficult to see that $h(Q)=5$ in this
case.
\end{exam}

\end{document}